\documentclass[12pt,oneside,english]{amsart}
\usepackage[T1]{fontenc}
\usepackage[latin9]{inputenc}
\usepackage{amsthm}
\usepackage{amssymb}
\usepackage{esint}

\makeatletter

\newcommand{\noun}[1]{\textsc{#1}}

\numberwithin{equation}{section}
\numberwithin{figure}{section}
\theoremstyle{plain}
\newtheorem{thm}{Theorem}
  \theoremstyle{plain}
  \newtheorem*{conjecture*}{Conjecture}
  \theoremstyle{plain}
  \newtheorem{lem}[thm]{Lemma}
  \theoremstyle{remark}
  \newtheorem*{rem*}{Remark}

\makeatother

\makeatother

\usepackage{babel}

\begin{document}

\title{Iterated resolvent estimates for power bounded matrices }

\author{Rachid Zarouf}
\begin{abstract}
We discuss analogs of the Kreiss resolvent condition for power bounded
matrices. We also explain how to extend it to analogs of the Hille-Yosida
condition. 
\end{abstract}
\maketitle

\part{Introduction}

\section{ Resolvent estimates of power bounded matrices}

Let $T\,:\left(\mathbb{C}^{n},\,\left|\cdot\right|\right)\mapsto\left(\mathbb{C}^{n},\,\left|\cdot\right|\right)$
be an operator acting on a finite dimensional Banach space. We suppose
that $T$ satisfies the following \textit{power boundedness condition}:

\global\long\def\theequation{${\rm PBC}$}
\begin{equation}
P(T)=\sup_{k\geq0}\left\Vert T^{k}\right\Vert _{E\rightarrow E}<\infty\,.\label{eq:}\end{equation}
 We denote by $\sigma(T)=\left\{ \lambda_{1},\,...,\,\lambda_{n}\right\} $
the spectrum of $T,$ $r(T)=\max_{i}\left|\lambda_{i}\right|$ its
spectral radius (which satisfies $r(T)\leq1$ since $P(T)<\infty$),
and $R(z,\, T)=(zId-T)^{-1}$ the resolvent of $T$ at point $z$,
for $z\in\mathbb{C}\setminus\sigma(T),$ $Id$ being the identity
operator. Our problem here is to {}``study'' the quantity $\left\Vert R(z,\, T)\right\Vert .$

Having a brief look at published papers on this subject one can notice
that $\left\Vert R(z,\, T)\right\Vert $ is

(1) sometimes associated with the quantity $\left(\left|z\right|-1\right),$
(see for instance {[}GZ, Kr, LeTr, Nev{]}) and

(2) sometimes with the quantity ${\rm dist}\left(z,\,\sigma(T)\right),$
(see for instance {[}DS, Sand, SpSt, Z3{]}).

\begin{flushleft}
As regards point 1, we are lead to the so-called \textit{Kreiss resolvent
condition} (KRC). In the same spirit, point 2 leads to a strong version
of the classical KRC, see for instance {[}Sp1, Section 5{]}. 
\par\end{flushleft}

Section 2 below, deals with point (1) and with the Kreiss resolvent
condition in general. The classical KRC is recalled in Paragraph 2.1
whereas we develop in Paragraph 2.2 a natural extension of this classical
KRC. In Section 3, we deal with point (2) and recall a result by B.
Simon and E.B. Davies {[}DS{]} which we sharpen in {[}Z3{]}. Finally,
Section 4 is devoted to the so-called \textit{Hille-Yosida} or \textit{iterated
resolvent condition}, (which deals with estimates of powers of $R(z,\, T)$).

Each of our estimates are consequences (in Paragraph 2.1, in Section
3 and in Section 4) of Bernstein-type inequalities for rational functions
(BTIRF). Moreover, the so-called \textit{Kreiss Matrix Theorem} {[}Kr{]}
can be proved using a BTIRF (see {[}LeTr, Sp{]}). That is the reason
why before starting Section 2, we recall in Paragraph 1.2 below, the
definition of a BTIRF.

\section{Bernstein-type inequalities for rational functions in Hardy spaces}

Let $\mathcal{P}_{n}$ be the complex space of analytic polynomials
of degree less or equal than $n\ge1$. Let $\mathbb{D}=\left\{ z\in\mathbb{C}\,:\left|z\right|<1\right\} $
be the standard unit disc of the complex plane and $\overline{\mathbb{D}}$
its closure. Given $r\in[0,\,1),$ we define\[
\mathcal{R}_{n,\, r}=\left\{ \frac{p}{q}\,:\; p,\, q\in\mathcal{P}_{n},\;{\rm d}^{\circ}p<{\rm d}^{\circ}q,\; q(\zeta)=0\Longrightarrow\zeta\notin\frac{1}{r}\mathbb{D}\right\} ,\]
(where ${\rm d}^{\circ}p$ means the degree of any $p\in\mathcal{P}_{n}$),
the set of all rational functions in $\mathbb{D}$ of degree less
or equal than $n\ge1$, having at most $n$ poles all outside of $\frac{1}{r}\mathbb{D}.$
Notice that for $r=0$, we get $\mathcal{R}_{n,\,0}=\mathcal{P}_{n-1}$.

\textit{Statement of the problem}. Generally speaking, given two Banach
spaces $X$ and $Y$ of holomorphic functions on the unit disc $\mathbb{D},$
we are searching for the {}``best possible'' constant $\mathcal{C}_{n,\, r}(X,\, Y)$
such that \[
\left\Vert f'\right\Vert _{X}\leq\mathcal{C}_{n,\, r}(X,\, Y)\left\Vert f\right\Vert _{Y},\]
 for all $f\in\mathcal{R}_{n,\, r}.$

From now on, the letter $c$ denotes a positive constant that may
change from one step to the next. For two positive functions $a$
and $b$, we say that $a$ is dominated by $b$, denoted by $a=O(b),$
if there is a constant $c>0$ such that $a\leq cb;$ and we say that
$a$ and $b$ are equivalent, denoted by $a\asymp b$, if both $a=O(b)$
and $b=O(a)$ hold.

The spaces $X$ and $Y$ considered here are nothing but the standard
Hardy spaces of the unit disc $\mathbb{D},$ $H^{p}=H^{p}(\mathbb{D}),$
$1\leq p\leq\infty$,\[
H^{p}=\left\{ f=\sum_{k\geq0}\hat{f}(k)z^{k}:\:\left\Vert f\right\Vert _{H^{p}}^{p}=\sup_{0\leq r<1}\int_{\mathbb{T}}\left|f(rz)\right|^{p}{\rm d}m(z)<\infty\right\} ,\]
 $m$ being the normalized Lebesgue measure on the unit circle $\mathbb{T}=\left\{ z\in\mathbb{C}\,:\left|z\right|=1\right\} .$

\part{Bernstein-type estimates }

The following Theorem belongs to K. Dyakonov but in {[}Dy{]}, the
constant $c_{p}$ is not explicitely given.
\begin{thm}
\begin{flushleft}
Let $p\in[1,\,\infty].$ We have\[
\mathcal{C}_{n,\, r}\left(H^{1},\, H^{p}\right)\leq c_{p}\frac{n}{(1-r)^{\frac{1}{p}}},\]
where $c_{p}=(1+r)^{\frac{1}{p}}.$
\par\end{flushleft}\end{thm}
\begin{conjecture*}
We suspect the constant $(1+r)^{\frac{1}{p}}$ in (10) to be asymptotically
sharp as $n$ tends to $+\infty$. This asymptotic sharpness should
be proved using the same test function as in {[}Z1{]}. In particular,
it is proved in {[}Z1{]} that there exists a limit\[
\lim_{n\rightarrow\infty}\frac{\mathcal{C}_{n,\, r}\left(H^{2},\, H^{2}\right)}{n}=\frac{1+r}{1-r}.\]

\end{conjecture*}
Before giving the proof of Theorem A, we need to state three definitions
in which $\sigma=\left\{ \lambda_{1},\,...,\,\lambda_{n}\right\} \subset\mathbb{D}$\textit{
}is a finite subset of the unit disc.

\begin{flushleft}
\textbf{Definition 1.}\textit{ The Blashke product $B_{\sigma}$.}
We define the finite Blaschle product $B_{\sigma}$ corresponding
to $\sigma$ by 
\par\end{flushleft}

\def\theequation{${1}$}\begin{equation}
B_{\sigma}=\prod_{i=1}^{n}b_{\lambda_{i}},\label{eq:-6-1}\end{equation}
where $b_{\lambda}=\frac{\lambda-z}{1-\overline{\lambda}z}\:,$ is
the elementary Blaschke factor corresponding to $\lambda\in\mathbb{D}$. 

\begin{flushleft}
\textbf{Definition 2.}\textit{ The model space $K_{B_{\sigma}}$.
}We define $K_{B_{\sigma}}$ to be the $n$-dimensional space:
\par\end{flushleft}

\def\theequation{${2}$}\begin{equation}
K_{B_{\sigma}}=\left(B_{\sigma}H^{2}\right)^{\perp}=H^{2}\ominus B_{\sigma}H^{2}.\label{eq:-6}\end{equation}
\textbf{Definition 3.}\textit{ The differentiation operator on $K_{B_{\sigma}}$.
}Let $D$ be the operator of differentiation on $\left(K_{B_{\sigma}},\,\left\Vert \cdot\right\Vert _{H^{1}}\right):$

\def\theequation{${3}$}\begin{equation}
\begin{array}{c}
D:\:\left(K_{B_{\sigma}},\,\left\Vert \cdot\right\Vert _{H^{1}}\right)\rightarrow\left(L^{p},\,\left\Vert \cdot\right\Vert _{L^{p}}\right)\\
f\mapsto f',\end{array}.\label{eq:-6-2}\end{equation}

\begin{flushleft}
\textbf{Remark 4. }Let $p\in[1,\,\infty].$ We have
\par\end{flushleft}

\def\theequation{${4}$}\begin{equation}
C_{n,\, r}\left(H^{1},\, H^{p}\right)=\label{eq:-6-2-1}\end{equation}
\[
=\sup\left\{ \left\Vert D\right\Vert _{\left(K_{B_{\sigma}},\,\left\Vert \cdot\right\Vert _{H^{1}}\right)\rightarrow\left(L^{p},\,\left\Vert \cdot\right\Vert _{L^{p}}\right)}:\,1\leq{\rm card}\,\sigma\leq n,\,\left|\lambda\right|\leq r\:\forall\lambda\in\sigma\right\} .\]

\begin{flushleft}
\textit{Proof of Theorem A.} We know that $\mathcal{C}_{n,\, r}\left(H^{1},\, H^{\infty}\right)=n$
(see {[}LeTr{]}) on one hand, and on the other hand that $\mathcal{C}_{n,\, r}\left(H^{1},\, H^{1}\right)\leq n\frac{1+r}{1+r}$
(see {[}Ba{]}). Moreover, if $p\in[1,\,\infty],$ there exists $0\leq\theta\leq1$
such that $1/p=1-\theta$, and using the notation of the complex interpolation
theory between Banach spaces see {[}6, 17{]}, we have $\left[L^{1},\, L^{\infty}\right]_{\theta}=L^{p}$
(with equal norms). Applying this result with the differentiation
operator $D$, we get \[
C_{n,\, r}\left(H^{1},\, H^{p}\right)\leq C_{n,\, r}\left(H^{1},\, H^{1}\right)^{\frac{1}{p}}C_{n,\, r}\left(H^{1},\, H^{\infty}\right)^{1-\frac{1}{p}}=\]
\[
\leq\left(n\frac{1+r}{1+r}\right)^{\frac{1}{p}}n^{1-\frac{1}{p}}=n\left(\frac{1+r}{1+r}\right)^{\frac{1}{p}},\]
which completes the proof.
\par\end{flushleft}

\begin{flushright}
$\square$
\par\end{flushright}

\part{Resolvent estimates}

\section{Kreiss resolvent conditions}

\subsection{Known results: $\mathcal{C}_{n,\, r}\left(H^{1},\, H^{\infty}\right)$
and the classical KRC}

Leveque, Trefethen {[}LeTr{]} and Spijker {[}Sp{]} managed to establish
a link between the constants $C_{n,\, r}\left(H^{1},\, H^{\infty}\right)$
and the problem of resolvent estimates for power bounded matrices.
Given $T\,:\left(\mathbb{C}^{n},\,\left|\cdot\right|_{2}\right)\mapsto\left(\mathbb{C}^{n},\,\left|\cdot\right|_{2}\right)$
and $r=r(T)$, it is not difficult to construct a rational function
with $n$ poles all outside of $\frac{1}{r}\mathbb{D}:$ one can simply
take $u$ and $v$ two unit vectors of $\mathbb{C}^{n}$, set \[
f(z)=^{t}uR(z,\, T)v\]
and apply a Bernstein-type estimate to this rational function. Leveque,
Trefethen {[}LeTr{]} and Spijker {[}Sp{]} use the one of $C_{n,\, r}\left(H^{1},\, H^{\infty}\right).$ 

The classical KRC is satisfied if and only if (by definition)

\global\long\def\theequation{${\rm KRC}$}
\begin{equation}
\rho(T)=\sup_{\vert z\vert>1}\left(\vert z\vert-1\right)\left\Vert R(z,\, T)\right\Vert <\infty.\label{eq:}\end{equation}
 There exists a link between the conditions (KRC) and (PBC): they
are equivalent. Indeed, \[
\rho(T)\underbrace{\leq}_{(5)}P(T)\underbrace{\leq}_{(6)}en\rho(T),\]
 but we have to be careful: (5) is true for every power bounded operator
(not necessarily acting on a finite dimensional Banach space) and
is very easy to check (by a power series expansion of $R(z,\, T)),$
whereas (6) is much more difficult to verify and has been proved only
for the Hilbert norm $\vert\cdot\vert=\vert\cdot\vert_{2}.$ In fact,
the statement \[
({\rm KRC})\Longrightarrow({\rm PBC}),\]
 is known as\textit{ Kreiss Matrix Theorem} {[}Kr{]}. According to
Tadmor, it has been shown originally by Kreiss (1962) with the inequality
$P(T)\leq Cste\left(\rho(T)\right)^{n^{n}}$. It is useful in proofs
of stability theorems for finite difference approximations to partial
differential equations. Until 1991, the inequality of Kreiss has been
improved successively by Morton, Strang, Miller, Laptev, Tadmor, Leveque
and Trefethen {[}LeTr{]} with the inequality

\global\long\def\theequation{${7}$}
\begin{equation}
P(T)\leq2en\rho(T)),\label{eq:}\end{equation}
 which is a consequence of the Bernstein-type estimate (4) (also proved
in {[}LeTr{]} by Leveque and Trefethen):

\global\long\def\theequation{${8}$}
\begin{equation}
\mathcal{C}_{n,\, r}\left(H^{1},\, H^{\infty}\right)\leq2n,\label{eq:}\end{equation}
 and finally Spijker {[}Sp2{]} with the inequality (5)

\global\long\def\theequation{${9}$}
\begin{equation}
P(T)\leq en\rho(T),\label{eq:}\end{equation}
(in which the constant $en$ is sharp), which is again a consequence
of the Bernstein-type estimate (6) (also proved in {[}Sp2{]} by Spijker):

\global\long\def\theequation{${10}$}
\begin{equation}
\mathcal{C}_{n,\, r}\left(H^{1},\, H^{\infty}\right)=n.\label{eq:}\end{equation}
 Notice that the problem of estimating $\mathcal{C}_{n,\, r}\left(H^{1},\, H^{\infty}\right)$
was already studied by Dolzhenko {[}Dol{]} (see also {[}Pek{]}, p.560
- inequality (11)). He proved that

\global\long\def\theequation{${11}$}
\begin{equation}
C_{n,\, r}\left(H^{1},\, H^{\infty}\right)\leq cn,\label{eq:}\end{equation}
 where $c$ is a numerical constant ($c$ is not explicitely given
in {[}Dol{]}).

\subsection{Application of the estimate of $\mathcal{C}_{n,\, r}\left(H^{1},\, H^{p}\right)$
to new resolvent estimates\textmd{ }}

Given $T\,:\left(\mathbb{C}^{n},\,\left|\cdot\right|\right)\mapsto\left(\mathbb{C}^{n},\,\left|\cdot\right|\right)$
and $r=r(T)$, we will construct another rational function $f$ (than
the one used by Leveque, Trefethen {[}LeTr{]} and Spijker {[}Sp{]})
with $n$ poles all outside of $\frac{1}{r}\mathbb{D}$ and of course
apply Theorem A to $f$.

\subsubsection{The main tools}

Let us define a special family of rational functions associated with
the spectrum of $T$ $\sigma(T)=\left\{ \lambda_{1},\,...,\,\lambda_{n}\right\} .$

\begin{flushleft}
\textbf{Definition 5.}\textit{ Malmquist family. }For $k\in[1,\, n]$,
we set $f_{k}=\frac{1}{1-\overline{\lambda_{k}}z},$ and define the
family $\left(e_{k}\right)_{1\leq k\leq n}$, (which is known as Malmquist
basis, see {[}13, p.117{]}), by
\par\end{flushleft}

\def\theequation{${12}$}\begin{equation}
e_{1}=\frac{f_{1}}{\left\Vert f_{1}\right\Vert _{2}}\,\,\,\mbox{and}\,\,\, e_{k}=\left({\displaystyle \prod_{j=1}^{k-1}}b_{\lambda_{j}}\right)\frac{f_{k}}{\left\Vert f_{k}\right\Vert _{2}}\,,\label{eq:-7}\end{equation}
for $k\in[2,\, n]$; we have $\left\Vert f_{k}\right\Vert _{2}=\left(1-\vert\lambda_{k}\vert^{2}\right)^{-1/2}.$

\begin{flushleft}
\textbf{Definition 7.}\textit{ The orthogonal projection ~$P_{B_{\sigma}}$on
$K_{B_{\sigma}}.$ }We define $P_{B_{\sigma}}$ to be the orthogonal
projection of $H^{2}$ on its $n$-dimensional subspace $K_{B_{\sigma}}.$
\par\end{flushleft}

\begin{flushleft}
\textbf{Remark 8.} The Malmquist family $\left(e_{k}\right)_{1\leq k\leq n}$
corresponding to $\sigma$ is an orthonormal basis of $K_{B_{\sigma}}.$
In particular,
\par\end{flushleft}

\def\theequation{${14}$}\begin{equation}
P_{B_{\sigma}}=\sum_{k=1}^{n}\left(\cdot,\, e_{k}\right)_{H^{2}}e_{k}\,,\label{eq:-7}\end{equation}

\begin{flushleft}
where $\left(\cdot,\,\cdot\right)_{H^{2}}$ means the scalar product
on $H^{2}$.
\par\end{flushleft}

\subsubsection{An interpolation problem in the Wiener algebra}

Here, we transform our problem of resolvent estimates into an interpolation
one in the Wiener algebra.

\begin{flushleft}
\textbf{Definition 9 .} Let $W$ be the Wiener algebra of absolutely
converging Fourier series:\[
W=\left\{ f=\sum_{k\geq0}\hat{f}(k)z^{k}:\:\left\Vert f\right\Vert _{W}=\sum_{k\geq0}\left|\hat{f}(k)\right|<\infty\right\} .\]

\par\end{flushleft}
\begin{lem}
\begin{flushleft}
Let $T\,:\left(\mathbb{C}^{n},\,\left|\cdot\right|\right)\mapsto\left(\mathbb{C}^{n},\,\left|\cdot\right|\right)$
be a power bounded operator and $\sigma=\sigma(T)=\left\{ \lambda_{1},\,...,\,\lambda_{n}\right\} $
its spectrum. Let also $l=1,\,2,\,...$, and $\lambda\in\mathbb{C}\setminus\overline{\mathbb{D}}$.
Then, 
\par\end{flushleft}

\begin{flushleft}
\[
\left\Vert R^{l}(\lambda,T)\right\Vert \leq P(T)\frac{1}{\left|\lambda\right|^{l}}\left\Vert P_{B_{\sigma}}\left(k_{1/\bar{\lambda}}\right)^{l}\right\Vert _{W}.\]

\par\end{flushleft}\end{lem}
\begin{proof}
\begin{flushleft}
First of all, 
\par\end{flushleft}

\begin{flushleft}
\[
\left\Vert R^{l}(\lambda,T)\right\Vert \leq P(T)\left\Vert \left(\frac{1}{\lambda-z}\right)^{l}\right\Vert _{W/B_{\sigma}W},\]
(see {[}3{]} Theorem 3.24, p.31), where\[
\left\Vert \left(\frac{1}{\lambda-z}\right)^{l}\right\Vert _{W/B_{\sigma}W}=\inf\left\{ \left\Vert f\right\Vert _{W}:\, f\left(\lambda_{j}\right)=\frac{1}{\left(\lambda-\lambda_{j}\right)^{l}},\, j=1..n\right\} .\]
We obtain the result since the function $f=P_{B_{\sigma}}\left(\frac{1}{\lambda^{l}}k_{1/\bar{\lambda}}\right)^{l}$
satisfies $f-\left(\frac{1}{\lambda-z}\right)^{l}\in B_{\sigma}W,\,\forall\, j=1..n$.
\par\end{flushleft}
\end{proof}

\subsubsection{Consequence: a possible extension of the classical KRC}

In this paragraph, we focus on the above inequality (1) and assume
that $\alpha\in(0,\,1).$ We notice that in this case, $\left(\vert z\vert-1\right)^{\alpha}\gg\vert z\vert-1$
as $\vert z\vert\rightarrow1^{+}.$ As a consequence, we ask the following
\textbf{question}: is it possible to find a constant $C_{\alpha}>0$
such that

\global\long\def\theequation{${9}$}
\begin{equation}
\left\Vert R(z,\, T)\right\Vert \leq C_{\alpha}\frac{P(T)}{\left(\vert z\vert-1\right)^{\alpha}},\label{eq:}\end{equation}
 for all $\vert z\vert>1$ and for all $T$?

The\textbf{ answer} is {}``No'' if $r(T)=1$ and {}``Yes'' if
$r(T)<1$ but with a constant $C_{\alpha}=C_{\alpha}(n,\, r(T))$
which depends on the size $n$ of $T$ and on its spectral radius
$r(T).$

More precisely, we define\[
\rho_{\alpha}(T)=\sup_{\vert z\vert>1}\left(\vert z\vert-1\right)^{\alpha}\left\Vert R(z,\, T)\right\Vert ,\]
 and prove the following theorem.
\begin{thm}
\begin{flushleft}
\textbf{ }Let $\alpha\in(0,\,1)$.
\par\end{flushleft}

\begin{flushleft}
(i) The condition: 
\par\end{flushleft}

\global\long\def\theequation{${{\rm KRC}_{\alpha}}$}
\begin{equation}
\rho_{\alpha}(T)<\infty,\label{eq:}\end{equation}
 is satisfied if and only if $r(T)<1.$ Moreover, in this case 

\global\long\def\theequation{${10}$}
\begin{equation}
\rho_{\alpha}(T)\leq C_{\alpha}\left(n,\, r(T)\right)P(T),\label{eq:-5}\end{equation}
with

\global\long\def\theequation{${11}$}
\begin{equation}
C_{\alpha}\left(n,\, r(T)\right)=K_{\alpha}\frac{n}{\left(1-r(T)\right)^{1-\alpha}},\label{eq:}\end{equation}
 where $K_{\alpha}$ is a constant depending only on $\alpha.$ 

\begin{flushleft}
(ii) Asymptotic sharpness of (9) as $n$ tends to $\infty$ and $r$
tends to 1: there exists a contraction $A_{r}$ on the Hilbert space
$\left(\mathbb{C}^{n},\,\left|\cdot\right|_{2}\right)$ of spectrum
$\{r\}$ such that
\par\end{flushleft}

\global\long\def\theequation{${12}$}
\begin{equation}
\liminf_{r\rightarrow1^{-}}(1-r)^{1-\alpha-\beta}\rho_{\alpha}\left(A_{r}\right)\geq\mbox{cot}\left(\frac{\pi}{4n}\right)\geq P(A_{r})\mbox{cot}\left(\frac{\pi}{4n}\right),\label{eq:-4}\end{equation}
 for all $\beta\in\left(0,\,1-\alpha\right).$

\begin{flushleft}
(iii) The analog of the Kreiss Matrix Theorem is satisfied with $\rho_{\alpha}(T):$
if $T\,:\left(\mathbb{C}^{n},\,\left|\cdot\right|_{2}\right)\mapsto\left(\mathbb{C}^{n},\,\left|\cdot\right|_{2}\right),$
where $\left|\cdot\right|_{2}$ is the Hilbert norm on $\mathbb{C}^{n},$
then
\par\end{flushleft}

\global\long\def\theequation{${13}$}
\begin{equation}
P(T)\leq en\rho_{\alpha}(T).\label{eq:-4-1}\end{equation}

\end{thm}
\begin{flushleft}
\textit{Comments on the proof of Theorem }3. 
\par\end{flushleft}

\begin{flushleft}
(a) As before, the estimate (11) of $C_{\alpha}$ is a consequence
of the Bernstein-type estimate (8) of Theorem A. More precisely, we
apply the BTIRF (8) with $p=\frac{1}{1-\alpha}$ so as to get the
estimate (11) of $C_{\alpha}$.\textit{ }
\par\end{flushleft}

\begin{flushleft}
(b) In inequality (12), $\beta$ is a {}``parasit'' parameter which
we can probably avoid.
\par\end{flushleft}

\begin{flushleft}
(c) The proof of (12) is the same as the one of (2) due to Leveque
and Trefethen {[}LeTr{]}. 
\par\end{flushleft}
\begin{rem*}
Considering $n\times n$ bidiagonal matrices of the form \[
T=\left(\begin{array}{ccccc}
\lambda_{1} & 0 & 0 & \cdots & 0\\
2 & \lambda_{2} & 0 & \ddots & \vdots\\
0 & \ddots & \ddots & \ddots & 0\\
\vdots & \ddots & \ddots & \lambda_{n-1} & 0\\
0 & \cdots & 0 & 2 & \lambda_{n}\end{array}\right),\:\left|\lambda_{i}\right|<1,\,\forall i=1,\,2,\,...,\]
 which {}``may be thought of as arising in the numerical solution
of an initial-boundary value'' as it is mentioned in {[}BoSp{]},
page 44, our estimates (8) gives better upper bounds (asymptotically
as $n\rightarrow\infty$ and/or $r(T)\rightarrow1^{-}$), of the quantity
$\left\Vert R(z,\, T)\right\Vert $ than the classical inequality
(1). 
\end{rem*}
\begin{flushleft}
\textit{Proof of Theorem 3}. (i) First of all, if $r(T)=1,$ then
$\rho_{\alpha}(T)=\infty$ because of the well-known inequality \[
\left\Vert R(\lambda,\, T)\right\Vert \geq\frac{1}{{\rm dist}(\lambda,\,\sigma(T))}.\]
 On the other hand, if $r(T)<1,$ then we apply Lemma 2 (the case
$l=1)$ combined with Hardy's inequality $\left\Vert f\right\Vert _{W}\leq\pi\left\Vert f'\right\Vert _{H^{1}}+\left|f(0)\right|$,
(see N. Nikolski, {[}4{]} p. 370 8.7.4 -(c)) to $g=P_{B_{\sigma}}\left(k_{1/\bar{\lambda}}\right).$
We obtain: \[
\left\Vert R(\lambda,T)\right\Vert \leq P(T)\frac{1}{\left|\lambda\right|}\left(\pi\left\Vert g'\right\Vert _{H^{1}}+\left|g(0)\right|\right).\]
Now we set $p=\frac{1}{1-\alpha}\in(1,\,+\infty),$ and get \[
\left\Vert R(\lambda,T)\right\Vert \leq P(T)\frac{1}{\left|\lambda\right|}\left(\pi C_{n,\, r}\left(H^{1},\, H^{p}\right)+1\right)\left\Vert g\right\Vert _{H^{p}}\leq\]
\[
\leq2P(T)\pi(1+r)^{\frac{1}{p}}\frac{n}{(1-r)^{\frac{1}{p}}}\left(1+\left(\frac{1-r}{1+r}\right)^{\frac{1}{p}}\frac{1}{\pi n}\right)\frac{1}{\left|\lambda\right|}\frac{1}{\left(1-\frac{1}{\left|\lambda\right|^{2}}\right)^{1-\frac{1}{p}}},\]
using Theorem A. Finally, \[
\frac{\left|\lambda\right|}{\left|\lambda\right|^{2-\frac{2}{p}}}\left(\left|\lambda\right|^{2}-1\right)^{1-\frac{1}{p}}\left\Vert R(\lambda,T)\right\Vert \leq2\left(\pi+1\right)(1+r)^{\frac{1}{p}}P(T)\frac{n}{(1-r)^{\frac{1}{p}}},\]
which means (since $\left|\lambda\right|>1$) \[
2^{1-\frac{1}{p}}\left(\left|\lambda\right|-1\right)^{1-\frac{1}{p}}\left\Vert R(\lambda,T)\right\Vert \leq\]
\[
\leq\left|\lambda\right|^{\frac{2}{p}-1}\left(\left|\lambda\right|^{2}-1\right)^{1-\frac{1}{p}}\left\Vert R(\lambda,T)\right\Vert \leq\]
\[
\leq2\left(\pi+1\right)(1+r)^{\frac{1}{p}}P(T)\frac{n}{(1-r)^{\frac{1}{p}}},\]
and \[
\rho_{\alpha}(T)\leq\left(\pi+1\right)(2(1+r))^{1-\alpha}\frac{n}{(1-r)^{1-\alpha}}P(T).\]

\par\end{flushleft}

(ii) Let $M_{n}$ be the $n\times n$ nilpotent Toeplitz matrix defined
by \[
M_{n}=\left(\begin{array}{ccccc}
0 & 1 & 0 & . & 0\\
. & 0 & 1 & . & .\\
. & . & . & . & 0\\
0 & . & . & 0 & 1\\
0 & . & . & . & 0\end{array}\right).\]
 Let also $f_{r}=\frac{z+r}{1+rz}$ with $r\in(0,\,1).$ For every
$\lambda\in\mathbb{C},$ 

\[
\lambda-f_{r}=\lambda-\frac{z+r}{1+rz}=\frac{\lambda-r-z(1-\lambda r)}{1+rz}.\]
So if we set $A_{r}=f_{r}(M_{n}),$ then \[
\left(\lambda I_{n}-A_{r}\right)^{-1}=\left(\lambda I_{n}-f_{r}(M_{n})\right)^{-1}=\left((\lambda-f_{r})(M_{n})\right)^{-1}\]
 \[
=(\lambda-f_{r})^{-1}(M_{n})=\left(\frac{1+rz}{\lambda-r-z(1-\lambda r)}\right)(M_{n})=\frac{1}{\lambda-r}\left(\frac{1+rz}{1+z\frac{\lambda r-1}{\lambda-r}}\right)(M_{n})\,.\]
We suppose $\lambda\in\mathbb{R}$, $\lambda>1$ and set $\nu=\frac{\lambda r-1}{\lambda-r},$
which means $\lambda=\frac{1-r\nu}{r-\nu}.$ Then\[
\lambda-1=\frac{(1-r)(1+\nu)}{r-\nu}\;\mbox{and}\;\lambda-r=\frac{1-r^{2}}{r-\nu}\,.\]
In particular,\[
\frac{\left(\left|\lambda\right|-1\right)^{\alpha}}{\left|\lambda-r\right|}=\frac{\left(\lambda-1\right)^{\alpha}}{\lambda-r}=\frac{1}{1+r}\frac{\sqrt{(1-r)(1+\nu)}}{\sqrt{r-\nu}}\frac{r-\nu}{1-r}=\]
\[
=\frac{1}{1+r}\sqrt{\frac{1+\nu}{1-r}}\sqrt{r-\nu}.\]
Now let $\alpha\in(0,\,1)$ and \[
\lambda=\lambda(r)=\frac{1+r-r(1-r)^{\alpha}}{1+r-(1-r)^{\alpha}},\]
then $\lambda>1$ and the corresponding $\nu=\nu(r)$ is given (after
calculation) by \[
\nu(r)=(1-r)^{\alpha}-1\,.\]
As a consequence taking $\lambda=\lambda(r)$ we get,\[
\mbox{sup}_{\left|\lambda\right|>1}\left(\left|\lambda\right|-1\right)^{\alpha}\left\Vert \left(\lambda I_{n}-f_{r}(M_{n})\right)^{-1}\right\Vert \geq\]
\[
\geq\frac{1}{1+r}\sqrt{\frac{1+\nu}{1-r}}\sqrt{r-\nu}\left\Vert \left(\frac{1+rz}{1+\nu z}\right)(M_{n})\right\Vert =\]
\[
=\frac{\sqrt{(\lambda+1)(r-\nu)}}{1+r}(1-r)^{\frac{\alpha}{2}-\frac{1}{2}}\left\Vert \left(\frac{1+rz}{1+\nu z}\right)(M_{n})\right\Vert ,\]
and \[
(1-r)^{\frac{1}{2}-\frac{\alpha}{2}}\mbox{sup}_{\left|\lambda\right|>1}\sqrt{\left|\lambda\right|^{2}-1}\left\Vert \left(\lambda I_{n}-f_{r}(M_{n})\right)^{-1}\right\Vert \geq\]
\[
\geq\frac{\sqrt{(\lambda+1)(r-\nu)}}{1+r}\left\Vert \left(\frac{1+rz}{1+\nu z}\right)(M_{n})\right\Vert ,\]
and taking finally the limit as $r$ tends to $1^{-},$ we get \[
\underline{\mbox{lim}}_{r\rightarrow1}(1-r)^{\frac{1}{2}-\frac{\alpha}{2}}\mbox{sup}_{\left|\lambda\right|>1}\sqrt{\left|\lambda\right|^{2}-1}\left\Vert \left(\lambda I_{n}-f_{r}(M_{n})\right)^{-1}\right\Vert \geq\]
\[
\geq\frac{\sqrt{(1+1)(1+1)}}{1+1}\left\Vert \left(\frac{1+z}{1-z}\right)(M_{n})\right\Vert =\]
\[
=\left\Vert \left(\frac{1+z}{1-z}\right)(M_{n})\right\Vert =\left\Vert \left(\begin{array}{ccccc}
1 & 2 & . & . & 2\\
 & 1 & 2 & . & 2\\
 &  & . & . & .\\
 &  &  & 1 & 2\\
 &  &  &  & 1\end{array}\right)\right\Vert =\mbox{cot}\left(\frac{\pi}{4n}\right)\,,\]
 see {[}DS{]} , Theorem 2 - p.4 for the last equality. 

(iii) The proof is the same as in {[}LeTr, page 4{]} bu instead of
taking the contour $\Gamma$ of integration : $\left|z\right|=1+\frac{1}{k+1}$,
(path on which $z^{k+1}\leq e$), we take the contour $\Gamma_{\alpha}$
of integration : $\left|z\right|=1+\frac{1}{(k+1)^{\alpha}}$, (path
on which $z^{k+1}\leq e$), we take

\begin{flushright}
$\square$
\par\end{flushright}

\subsubsection{A possible extension of the classical Hille-Yosida condition}

Here we extend the result of Theorem 3 (using Lemma 2 with $l\geq1)$,
(in subsection 3.2.3), to iterated resolvent condition. The so-called
\textit{Hille-Yosida} or \textit{iterated resolvent condition} is
recalled in {[}BoSp{]} (Section 3.3, statement (3.4.a)):

\global\long\def\theequation{${\rm HYC}$}
\begin{equation}
r(T)\leq1\:{\rm and}\:\rho^{k}(T)<\infty,\:\forall k=1,\,2,\,..,\label{eq:-2-1-1}\end{equation}
 where \[
\rho^{k}(T)=\sup_{\vert z\vert>1}\left(\vert z\vert-1\right)^{k}\left\Vert R^{k}(z,\, T)\right\Vert \]
 and $R^{k}(z,\, T)=(zId-T)^{-k}.$ Now in the same spirit as in Paragraph
3.2.3 it is natural to consider the quantity:\[
\rho_{\alpha}^{k}(T)=\sup_{\vert z\vert>1}\left(\vert z\vert-1\right)^{\alpha+k-1}\left\Vert R^{k}(z,\, T)\right\Vert \]

\begin{thm}
Let $k\geq1$ and $\alpha\in(0,\,1).$ We have 

\global\long\def\theequation{${15}$}
\begin{equation}
\rho_{\alpha}^{k}(T)\leq C_{\alpha,\, k}\left(n,\, r(T)\right)P(T),\label{eq:-2-3}\end{equation}
 where

\global\long\def\theequation{${16}$}
\begin{equation}
C_{\alpha,\, k}\left(n,\, r(T)\right)=K_{\alpha,\, k}\frac{n^{k}}{\left(1-r(T)\right)^{1-\alpha}},\label{eq:-2-2-2}\end{equation}
 $K_{\alpha,\, k}$ being a constant depending on $\alpha$ and $k$
only.\end{thm}
\begin{proof}
To be written.
\end{proof}

\part{Hille-Yosida and strong iterated resolvent conditions}

\section{Introduction}

\subsection{Statement of the problem}

Let $T\,:\left(\mathbb{C}^{n},\,\left|\cdot\right|\right)\mapsto\left(\mathbb{C}^{n},\,\left|\cdot\right|\right)$
be an operator acting on a finite dimensional Banach space. We suppose
that $T$ satisfies the following \textit{power boundedness condition}:

\def\theequation{${\rm PBC}$}\begin{equation}
P(T)=\sup_{k\geq0}\left\Vert T^{k}\right\Vert _{E\rightarrow E}<\infty.\label{eq:}\end{equation}
 We denote by $\sigma(T)=\left\{ \lambda_{1},\,...,\,\lambda_{n}\right\} $
the spectrum of $T,$ $r(T)=\max_{i}\left|\lambda_{i}\right|$ its
spectral radius (which satisfies $r(T)\leq1$ since $P(T)<\infty$),
and \[
R^{l}(\lambda,\, T)=(\lambda Id-T)^{-l},\, l=1,\,2,\,...,\]
 the resolvent of $T$ at point $\lambda$, for $\lambda\in\mathbb{C}\setminus\sigma(T),$
$Id$ being the identity operator. Our problem here is to find an
upper bound for the quantity $\left\Vert R^{l}(\lambda,\, T)\right\Vert $
in terms of $P(T),$ the size $n$ of the matrix $T$ and the quantity
${\rm dist}\left(\lambda,\,\sigma(T)\right)=\inf_{i}\left|\lambda-\lambda_{i}\right|$
as it is given in {[}DS{]} and in {[}Z3{]} for the case $l=1.$

\subsection{The classical Strong Hille-Yosida and Kreiss resolvent conditions}

Let $\mathbb{D}=\left\{ \lambda\,:\;\left|\lambda\right|<1\right\} $
be the unit disc of the complex plane and $\overline{\mathbb{D}}=\left\{ \lambda\,:\;\left|\lambda\right|\leq1\right\} $
its closure. We have to recall that according to {[}SpSt{]}, if $L$
is a positive constant and $W$ is a subset of the closed unit disc
$\overline{\mathbb{D}},$ the \textit{Strong Kreiss resolvent condition
}with respect to $W$ with constant $L,$ is the following:

\def\theequation{${\left({\rm SKRC}\right)_{W,\, L}}$}\begin{equation}
\left\Vert R(\lambda,\, T)\right\Vert \leq\frac{L}{{\rm dist}\left(\lambda,\, W\right)},\:\forall\lambda\in\mathbb{C}\setminus W.\label{eq:-3}\end{equation}
In the same spirit, according to {[}Sand{]} if $L$ is a positive
constant and $W$ is a subset of the closed unit disc $\overline{\mathbb{D}},$
the \textit{Strong Hille-Yosida resolvent condition} with respect
to $W$ with constant $L,$ is:

\def\theequation{${\left({\rm SHYRC}\right)_{W,\, L}}$}\begin{equation}
\left\Vert R^{l}(\lambda,\, T)\right\Vert \leq\frac{L}{\left({\rm dist}\left(\lambda,\, W\right)\right)^{l}},\:\forall\lambda\in\mathbb{C}\setminus W,\:\forall l=1,\,2,\,...\label{eq:-3-1}\end{equation}

\subsection{Definitions of our constants}

Here, we consider the case $W=\sigma(T)$ and define the corresponding
quantities\[
\rho^{strong}(T)=\sup_{\vert\lambda\vert\geq1}{\rm dist}\left(\lambda,\,\sigma(T)\right)\left\Vert R(z,\, T)\right\Vert ,\]
and more generally, if $l\geq1,$ \[
\rho^{strong,\, l}(T)=\sup_{\vert\lambda\vert\geq1}\left({\rm dist}\left(\lambda,\,\sigma(T)\right)\right)^{l}\left\Vert R^{l}(z,\, T)\right\Vert .\]
Our aim is to find upper estimates for $\rho^{strong}(T)$ and $\rho^{strong,\, l}(T)$
in terms of $P(T),$ $n$ and $l$.

\subsection{Known results, the case $l=1$: a strong version of the KRC}

It is important to recall that according to {[}SpSt{]}, page 78, if
$L$ is a positive constant and $W$ is a subset of the closed unit
disc $\overline{\mathbb{D}},$ the \textit{Strong Kreiss resolvent
condition with respect to $W$ with constant $L,$ }is the following:

\global\long\def\theequation{${\left({\rm SKRC}\right)_{W,\, L}}$}
\begin{equation}
\left\Vert R(z,\, T)\right\Vert \leq\frac{L}{{\rm dist}\left(z,\, W\right)},\:\forall z\in\mathbb{C}\setminus W.\label{eq:-3-2}\end{equation}
As a consequence, the question: {}`` what happens if we replace {}``$\left(\vert z\vert-1\right)$''
by {}``${\rm dist}(z,\,\sigma(T))$'' in Paragraph 1.1?'' can be
interpreted using the $\left({\rm SKRC}\right)_{W,\, L}$ with\[
W=\sigma(T).\]
 Dealing with the above question, we define the quantity\[
\rho^{strong}(T)=\sup_{\vert z\vert\geq1}{\rm dist}\left(z,\,\sigma(T)\right)\left\Vert R(z,\, T)\right\Vert ,\]
 which satisfies the inequality $\rho^{strong}(T)\geq\rho(T),$ since
$r(T)\leq1.$ Obviously, the condition

\global\long\def\theequation{${{\rm KRC}_{strong}}$}
\begin{equation}
\rho^{strong}(T)<\infty,\label{eq:-8}\end{equation}
 implies the classical KRC.

Notice that the quantity $\rho^{strong}(T)$ was already considered
and studied by B. Simon and E.B. Davies {[}DS{]} for contractions
and power bounded matrices on finite dimensional Hilbert spaces, (for
power bounded matrices only, see (13) below).

More precisely, they proved in {[}DS{]} among other things the following
result.
\begin{thm}
If $\vert\cdot\vert=\vert\cdot\vert_{2}$ is the Hilbert norm on $\mathbb{C}^{n},$
then

\global\long\def\theequation{${17}$}
\begin{equation}
\left\Vert R(z,\, T)\right\Vert \leq\left(\frac{3n}{{\rm dist}(z,\,\sigma(T))}\right)^{3/2}P(T),\label{eq:-8}\end{equation}
 for all $\vert z\vert\geq1,\, z\notin\sigma(T).$ 
\end{thm}
They suspect in {[}DS{]} that the power 3/2 is not sharp. In {[}Z3{]},
we improve their result (earning a square root at the denominator
of the above inequality) and prove the following theorem. 
\begin{thm}
Let $\vert\cdot\vert$ be a Banach norm on $\mathbb{C}^{n},$ then 
\end{thm}
\global\long\def\theequation{${18}$}
\begin{equation}
\rho^{strong}(T)\leq\left(\frac{5\pi}{3}+2\sqrt{2}\right)n^{3/2}P(T).\label{eq:-8}\end{equation}

\begin{proof}
See {[}Z3{]}.
\end{proof}
However, we still feel that the constant $n^{3/2}$ is not sharp ($n$
being probably the sharp one). One of the most important tool used
in order to prove the above inequality is again a BTIRF involving
the Hardy spaces $H^{1}$ and $H^{2}.$ \begin{rem*} We can say that
sharpening their result in {[}Z3{]} (see (14), below), we proved a
\textit{{}``unilateral version''} of the\textit{ Strong Kreiss resolvent
condition with respect to $\sigma(T)$ with the constant $L=\left(\frac{5\pi}{3}+2\sqrt{2}\right)n^{3/2}P(T),$
}for a power bounded matrix $T$, ($P(T)<\infty)$. The word \textit{{}``unilateral''}
is because of the fact that in the definition of $\rho^{strong}(T)$,
we take the supremum outside of the open unit disc $\mathbb{D}$ and
not also inside (obviously with the condition $z\in\mathbb{D}\setminus\sigma(T)$).
Why ? Because it is explained in {[}DS{]} page 3, see (1.14), (1.15),
that for $z\in\mathbb{D}\setminus\sigma(T)$, the quantity $\left\Vert R(z,\, T)\right\Vert $
may increase exponentially and can not be bounded by better than ${\rm dist}\left(z,\,\sigma(T)\right)^{-n}$
at least for the nilpotent Jordan block $T=N$ and for $z\in\mathbb{C}\setminus\{0\}$
(and of course $z$ close enough to the origin). \end{rem*}

\section{An extension of Theorem 6}

Here, we generalize the result of {[}Z3{]} and prove a \textit{{}``unilateral
version''} of the\textit{ Strong Hille-Yosida resolvent condition
}with respect to $\sigma(T)$ with a constant $L$\textit{ }for a
power bounded matrix $T$. As for the case $l=1$, this constant $L$
will depend on $P(T)$ and $n={\rm card}\;\sigma(T)$ but also on
$l.$ 

More precisely, we prove the following theorem. 
\begin{thm}
\begin{flushleft}
Let $l\geq1.$ We have
\par\end{flushleft}

\begin{flushleft}
\def\theequation{${19}$}\begin{equation}
\rho^{strong,\, l}(T)\leq C_{l}\left(n\right)P(T),\label{eq:-1}\end{equation}
where
\par\end{flushleft}

\def\theequation{${20}$}\begin{equation}
C_{l}\left(n\right)=K_{l}n^{l+\frac{1}{2}},\label{eq:-2-2-1}\end{equation}
$K_{l}$ being a constant depending $l$ only. 
\end{thm}
We first give two lemmas.

\subsection{A maximum principle. }

The following Lemma is a generalization of the Lemma of {[}Z3{]}. 
\begin{lem}
\begin{flushleft}
Let $C_{l}\left(n\right)>0$ such that for every operator $T$ acting
on $\left(\mathbb{C}^{n},\,\left|\cdot\right|\right)$ with spectrum
$\sigma(T)$, the following condition is satisfied : \[
\left\{ \begin{array}{c}
P(T)<\infty\\
\sigma(T)\subset\mathbb{D}\end{array}\right.\]
implies \[
\left[\forall\lambda_{\star}\:{\rm such}\,{\rm that}\:\left|\lambda_{\star}\right|=1,\:\left({\rm dist}\left(\lambda_{\star},\,\sigma(T)\right)\right)^{l}\left\Vert R^{l}\left(\lambda_{\star},T\right)\right\Vert \leq C_{l}\left(n\right)P(T)\right].\]
Then, \[
\rho^{strong,\, l}(T)\leq C_{l}\left(n\right)P(T),\]
for all power bounded operator $T$ acting on $\left(\mathbb{C}^{n},\,\left|\cdot\right|\right).$ 
\par\end{flushleft}\end{lem}
\begin{proof}
\begin{flushleft}
Let $\lambda$ such that $\left|\lambda\right|>1$. $\lambda$ can
be written as $\lambda=\rho\lambda_{\star}$ with $\rho>1$ and $\left|\lambda_{\star}\right|=1.$
We set $T_{\star}=\frac{1}{\rho}T$. Under these conditions, $P(T^{\star})\leq P(T)$
and $\sigma(T_{\star})=\frac{1}{\rho}\sigma(T)\subset\mathbb{D}$.
As a result, \[
\left({\rm dist}\left(\lambda_{\star},\,\sigma\left(T_{\star}\right)\right)\right)^{l}\left\Vert R^{l}\left(\lambda_{\star},T_{\star}\right)\right\Vert \leq C_{l}\left(n\right)P(T),\]
which can also be written as $\rho^{l}\left({\rm dist}\left(\lambda_{\star},\,\sigma\left(T_{\star}\right)\right)\right)^{l}\left\Vert \rho^{-l}R^{l}\left(\lambda_{\star},\, T_{\star}\right)\right\Vert \leq C_{l}\left(n\right)P(T).$
It is now sufficient to notice that $\rho^{l}\left({\rm dist}\left(\lambda_{\star},\,\sigma\left(T_{\star}\right)\right)\right)^{l}=\left({\rm dist}\left(\lambda_{\star},\,\sigma(T)\right)\right)^{l}$
and that $\rho^{-l}R^{l}\left(\lambda_{\star},\, T_{\star}\right)=R^{l}(\lambda,\, T).$ 
\par\end{flushleft}
\end{proof}

\subsection{Derivatives of the Malmquist family}

\begin{flushleft}
We recall the definition of the family $\left(e_{k}\right)_{1\leq k\leq n}$,
(which is known as Malmquist basis, see above), by
\par\end{flushleft}

\def\theequation{${21}$}\begin{equation}
e_{1}=\frac{\left(1-\vert\lambda_{1}\vert^{2}\right)^{-1/2}}{1-\overline{\lambda_{1}}z}\,\,\,\mbox{and}\,\,\, e_{k}=\left({\displaystyle \prod_{j=1}^{k-1}}b_{\lambda_{j}}\right)\frac{\left(1-\vert\lambda_{k}\vert^{2}\right)^{-1/2}}{1-\overline{\lambda_{k}}z},\label{eq:}\end{equation}
for $k\in[2,\, n],$ were $b_{\lambda_{j}}=\frac{\lambda_{j}-z}{1-\overline{\lambda_{j}}z}$
is the elementary Blaschke factor corresponding to $\lambda_{j}$.

In the following lemma, we find an upper estimate for values of derivatives
of $e_{k}$ on the unit circle.
\begin{lem}
\begin{flushleft}
Let $j\geq0.$ There exists a constant $C_{j}>0$ depending only on
$j$ such that \[
\left|\left(e_{k}\right)^{(j)}\left(\lambda_{\star}\right)\right|\leq C_{j}\left(1-\left|\lambda_{k}\right|^{2}\right)^{\frac{1}{2}}\frac{k^{j}}{\left({\rm dist}\left(\lambda_{\star},\,\sigma\right)\right)^{j+1}},\]
for all $\lambda_{\star}\in\mathbb{T}.$
\par\end{flushleft}\end{lem}
\begin{proof}
We prove it by induction on $j$. Indeed, \[
\left(e_{k}\right)^{(j+1)}=\left(e_{k}^{'}\right)^{(j)},\]
but \[
e_{k}^{'}=\sum_{i=1}^{k-1}\frac{b_{\lambda_{i}}^{'}}{b_{\lambda_{i}}}e_{k}+\overline{\lambda_{k}}\frac{1}{\left(1-\overline{\lambda_{k}}z\right)}e_{k}.\]
This gives \[
\left(e_{k}\right)^{(j+1)}=\left(\left(\sum_{i=1}^{k-1}\frac{b_{\lambda_{i}}^{'}}{b_{\lambda_{i}}}+\frac{\overline{\lambda_{k}}}{1-\overline{\lambda_{k}}z}\right)e_{k}\right)^{(j)}=\]
\[
=\left(\left(\sum_{i=1}^{k-1}\left(-\frac{1}{\lambda_{i}-z}+\frac{\overline{\lambda_{i}}}{1-\overline{\lambda_{i}}z}\right)+\frac{\overline{\lambda_{k}}}{1-\overline{\lambda_{k}}z}\right)e_{k}\right)^{(j)}=\]
\[
=\left(\left(\sum_{i=1}^{k}\frac{\overline{\lambda_{i}}}{1-\overline{\lambda_{i}}z}-\sum_{i=1}^{k-1}\frac{1}{\lambda_{i}-z}\right)e_{k}\right)^{(j)}=\]
\[
=\left(\left(\sum_{i=1}^{k}\frac{\overline{\lambda_{i}}}{1-\overline{\lambda_{i}}z}\right)e_{k}\right)^{(j)}-\left(\left(\sum_{i=1}^{k-1}\frac{1}{\lambda_{i}-z}\right)e_{k}\right)^{(j)}=\]
\[
=\sum_{s=0}^{j}\left(\begin{array}{c}
j\\
s\end{array}\right)\left(\sum_{i=1}^{k}\frac{\overline{\lambda_{i}}}{1-\overline{\lambda_{i}}z}\right)^{(s)}\left(e_{k}\right)^{(j-s)}-\sum_{s=0}^{j}\left(\begin{array}{c}
j\\
s\end{array}\right)\left(\sum_{i=1}^{k-1}\frac{1}{\lambda_{i}-z}\right)^{(s)}\left(e_{k}\right)^{(j-s)}=\]
\[
=\sum_{s=0}^{j}\left(\begin{array}{c}
j\\
s\end{array}\right)s!\left(e_{k}\right)^{(j-s)}\sum_{i=1}^{k}\frac{\overline{\lambda_{i}}^{s+1}}{\left(1-\overline{\lambda_{i}}z\right)^{s+1}}-\sum_{s=0}^{j}\left(\begin{array}{c}
j\\
s\end{array}\right)s!\left(e_{k}\right)^{(j-s)}\sum_{i=1}^{k-1}\frac{1}{\left(\lambda_{i}-z\right)^{s+1}}=\]
\[
=\sum_{s=0}^{j}\left(\begin{array}{c}
j\\
s\end{array}\right)s!\left(e_{k}\right)^{(j-s)}\sum_{i=1}^{k-1}\left[\frac{\overline{\lambda_{i}}^{s+1}}{\left(1-\overline{\lambda_{i}}z\right)^{s+1}}-\frac{1}{\left(\lambda_{i}-z\right)^{s+1}}\right]+\sum_{s=0}^{j}\left(\begin{array}{c}
j\\
s\end{array}\right)s!\frac{\overline{\lambda_{k}}^{s+1}}{\left(1-\overline{\lambda_{k}}z\right)^{s+1}}\left(e_{k}\right)^{(j-s)}.\]
As a consequence,\[
\left(e_{k}\right)^{(j+1)}\left(\lambda_{\star}\right)=\sum_{s=0}^{j}\left(\begin{array}{c}
j\\
s\end{array}\right)s!\sum_{i=1}^{k-1}\left[\frac{\overline{\lambda_{i}}^{s+1}}{\left(1-\frac{\overline{\lambda_{i}}}{\overline{\lambda_{\star}}}\right)^{s+1}}-\frac{1}{\left(\lambda_{i}-\lambda_{\star}\right)^{s+1}}\right]\left(e_{k}\right)^{(j-s)}\left(\lambda_{\star}\right)+\]
\[
+\sum_{s=0}^{j}\left(\begin{array}{c}
j\\
s\end{array}\right)s!\frac{\overline{\lambda_{k}}^{s+1}}{\left(1-\frac{\overline{\lambda_{k}}}{\overline{\lambda_{\star}}}\right)^{s+1}}\left(e_{k}\right)^{(j-s)}\left(\lambda_{\star}\right).\]
First we notice that\[
\left|\frac{\overline{\lambda_{i}}^{s+1}}{\left(1-\frac{\overline{\lambda_{i}}}{\overline{\lambda_{\star}}}\right)^{s+1}}\right|=\left|\frac{1}{\overline{\lambda_{\star}}^{s+1}}\right|\left|\frac{\overline{\lambda_{i}}^{s+1}}{\left(1-\frac{\overline{\lambda_{i}}}{\overline{\lambda_{\star}}}\right)^{s+1}}\right|=\]
\[
=\left|\frac{\overline{\lambda_{i}}^{s+1}}{\left(\overline{\lambda_{\star}}-\overline{\lambda_{i}}\right)^{s+1}}\right|=\left|\frac{\lambda_{i}}{\lambda_{\star}-\lambda_{i}}\right|^{s+1}\leq\frac{1}{\left({\rm dist}\left(\lambda_{\star},\,\sigma\right)\right)^{s+1}},\]
for all $i=1,\,2,\,...$ Applying the induction hypothesis, we get\[
\left|\sum_{s=0}^{j}\left(\begin{array}{c}
j\\
s\end{array}\right)s!\sum_{i=1}^{k-1}\left[\frac{\overline{\lambda_{i}}^{s+1}}{\left(1-\frac{\overline{\lambda_{i}}}{\overline{\lambda_{\star}}}\right)^{s+1}}-\frac{1}{\left(\lambda_{i}-\lambda_{\star}\right)^{s+1}}\right]\left(e_{k}\right)^{(j-s)}\left(\lambda_{\star}\right)\right|\leq\]
\[
\leq\sum_{s=0}^{j}\left(\begin{array}{c}
j\\
s\end{array}\right)s!\sum_{i=1}^{k-1}\left[\left|\frac{\overline{\lambda_{i}}^{s+1}}{\left(1-\frac{\overline{\lambda_{i}}}{\overline{\lambda_{\star}}}\right)^{s+1}}\right|+\left|\frac{1}{\left(\lambda_{i}-\lambda_{\star}\right)^{s+1}}\right|\right]\left|\left(e_{k}\right)^{(j-s)}\left(\lambda_{\star}\right)\right|\leq\]
\[
\leq\sum_{s=0}^{j}\left(\begin{array}{c}
j\\
s\end{array}\right)s!\sum_{i=1}^{k-1}\left[\frac{1}{\left({\rm dist}\left(\lambda_{\star},\,\sigma\right)\right)^{s+1}}+\frac{1}{\left({\rm dist}\left(\lambda_{\star},\,\sigma\right)\right)^{s+1}}\right]\left|\left(e_{k}\right)^{(j-s)}\left(\lambda_{\star}\right)\right|\leq\]
\[
\leq2\sum_{s=0}^{j}\left(\begin{array}{c}
j\\
s\end{array}\right)s!\frac{1}{\left({\rm dist}\left(\lambda_{\star},\,\sigma\right)\right)^{s+1}}\sum_{i=1}^{k-1}C_{j-s}\left(1-\left|\lambda_{k}\right|^{2}\right)^{\frac{1}{2}}\frac{k^{j-s}}{\left({\rm dist}\left(\lambda_{\star},\,\sigma\right)\right)^{j-s+1}}=\]
\[
=2\left(1-\left|\lambda_{k}\right|^{2}\right)^{\frac{1}{2}}\sum_{s=0}^{j}\left(\begin{array}{c}
j\\
s\end{array}\right)s!C_{j-s}\sum_{i=1}^{k-1}\frac{k^{j-s}}{\left({\rm dist}\left(\lambda_{\star},\,\sigma\right)\right)^{j+2}}=\]
\[
=2\left(1-\left|\lambda_{k}\right|^{2}\right)^{\frac{1}{2}}\frac{1}{\left({\rm dist}\left(\lambda_{\star},\,\sigma\right)\right)^{j+2}}\sum_{s=0}^{j}\left(\begin{array}{c}
j\\
s\end{array}\right)s!C_{j-s}\sum_{i=1}^{k-1}k^{j-s}=\]
\[
=2\left(1-\left|\lambda_{k}\right|^{2}\right)^{\frac{1}{2}}\frac{1}{\left({\rm dist}\left(\lambda_{\star},\,\sigma\right)\right)^{j+2}}\sum_{s=0}^{j}\left(\begin{array}{c}
j\\
s\end{array}\right)s!C_{j-s}k^{j+1-s}\leq\]
\[
\leq2\left(1-\left|\lambda_{k}\right|^{2}\right)^{\frac{1}{2}}\max_{s}\left[\left(\begin{array}{c}
j\\
s\end{array}\right)s!C_{j-s}\right]\frac{1}{\left({\rm dist}\left(\lambda_{\star},\,\sigma\right)\right)^{j+2}}\sum_{s=0}^{j}k^{j+1-s}\leq\]
\[
\leq2\left(1-\left|\lambda_{k}\right|^{2}\right)^{\frac{1}{2}}\max_{s}\left[\left(\begin{array}{c}
j\\
s\end{array}\right)s!C_{j-s}\right]\frac{(j+1)k^{j+1}}{\left({\rm dist}\left(\lambda_{\star},\,\sigma\right)\right)^{j+2}}=\]
\[
=C_{j+1}\left(1-\left|\lambda_{k}\right|^{2}\right)^{\frac{1}{2}}\frac{k^{j+1}}{\left({\rm dist}\left(\lambda_{\star},\,\sigma\right)\right)^{j+2}},\]
where\[
C_{j+1}=2(j+1)\max_{0\leq s\leq j}\left[\left(\begin{array}{c}
j\\
s\end{array}\right)s!C_{j-s}\right].\]

\end{proof}
\begin{flushleft}
\textbf{Proof of the Theorem.} The proof repeates the scheme from
the Theorem of {[}Z3{]} excepted that this time, we replace the function
$\frac{1}{\lambda-z}$ by $\left(\frac{1}{\lambda-z}\right)^{l}$
. Let $T$ be an $n\times n$ matrix such that $P(T)<\infty,$ with
$\sigma(T)=\left\{ \lambda_{1},\,\lambda_{2},\,...,\,\lambda_{n}\right\} $
(including multilpicities). We define $B=\prod_{i}b_{\lambda_{i}}$
the finite Blaschke product corresponding to $\sigma(T)$. Then,
\par\end{flushleft}

\textit{\[
\left\Vert R^{l}(\lambda,T)\right\Vert \leq P(T)\left\Vert \left(\frac{1}{\lambda-z}\right)^{l}\right\Vert _{W/BW},\]
}

\noun{\large \vspace{0.05cm}
}{\large \par}

\begin{flushleft}
(see Lemma 2 above), where $W$ stands for the Wiener algebra of absolutely
converging Fourier series:\[
W=\left\{ f=\sum_{k\geq0}\hat{f}(k)z^{k}:\:\left\Vert f\right\Vert _{W}=\sum_{k\geq0}\left|\hat{f}(k)\right|<\infty\right\} ,\]
 and
\par\end{flushleft}

\textit{\[
\left\Vert \left(\frac{1}{\lambda-z}\right)^{l}\right\Vert _{W/BW}=\inf\left\{ \left\Vert f\right\Vert _{W}:\, f\left(\lambda_{j}\right)=\frac{1}{\left(\lambda-\lambda_{j}\right)^{l}},\, j=1..n\right\} .\]
}

\noun{\large \vspace{0.05cm}
}{\large \par}

\begin{flushleft}
We first suppose that $\vert\lambda\vert>1.$ Let $P_{B}$ be the
orthogonal projection of the Hardy space $H^{2}$ onto the model space
$K_{B}=H^{2}\Theta BH^{2}$. Since the function $\frac{1}{\lambda-z}$
is here replaced by $\left(\frac{1}{\lambda-z}\right)^{l},$ the function
$f=P_{B}\left(\frac{1}{\lambda}k_{1/\bar{\lambda}}\right)$ in {[}Z3{]}
is replaced by $f=P_{B}\left(\frac{1}{\lambda^{l}}k_{1/\bar{\lambda}}\right)^{l}$
which satisfies $f-\left(\frac{1}{\lambda-z}\right)^{l}\in BW,\,\forall\, j=1..n$.
In particular,
\par\end{flushleft}

\def\theequation{${22}$}\begin{equation}
\left\Vert \left(\frac{1}{\lambda-z}\right)^{l}\right\Vert _{W/BW}\leq\left\Vert \frac{1}{\lambda^{l}}P_{B}\left(k_{1/\bar{\lambda}}\right)^{l}\right\Vert _{W}.\label{eq:}\end{equation}

\begin{flushleft}
Moreover, 
\par\end{flushleft}

\[
P_{B}\left(k_{1/\bar{\lambda}}\right)^{l}=\sum_{k=1}^{n}\left(\left(k_{1/\bar{\lambda}}\right)^{l},\, e_{k}\right)_{H^{2}}e_{k}=\]
\[
=\sum_{k=1}^{n}\left(z^{l-1}\left(k_{1/\bar{\lambda}}\right)^{l},\, z^{l-1}e_{k}\right)_{H^{2}}e_{k}.\]
Now we notice that

\def\theequation{${23}$}\begin{equation}
\left(f,\; z^{l-1}\left(k_{\zeta}\right)^{l}\right)_{H^{2}}=\left(f,\;\left(\frac{{\rm d}}{{\rm d}\overline{\zeta}}\right)^{l-1}k_{\zeta}\right)_{H^{2}}=f^{(l-1)}(\zeta),\label{eq:}\end{equation}
for every $l\geq1,$ $f\in H^{2},$ $\zeta\in\mathbb{D}.$ Applying
(5) with $f=z^{l-1}e_{k},$ $\zeta=\frac{1}{\bar{\lambda}},$ we get
\[
\left(z^{l-1}\left(k_{1/\bar{\lambda}}\right)^{l},\, z^{l-1}e_{k}\right)_{H^{2}}=\overline{\left(z^{l-1}e_{k},\, z^{l-1}\left(k_{1/\bar{\lambda}}\right)^{l}\right)_{H^{2}}}=\]
\[
=\overline{\left(z^{l-1}e_{k}\right)^{(l-1)}(1/\bar{\lambda})}.\]

\[
\left(z^{t}e_{k}\right)^{(t)}=\sum_{j=0}^{t}\left(\begin{array}{c}
t\\
j\end{array}\right)\left(e_{k}\right)^{(j)}\left(z^{t}\right)^{(t-j)}=\]
\[
=\sum_{j=0}^{t}\left(\begin{array}{c}
t\\
j\end{array}\right)\frac{t!}{j!}z^{j}\left(e_{k}\right)^{(j)}.\]
We complete the proof using {[}Z3{]}. 

\begin{flushright}
$\square$
\par\end{flushright}

\vspace{0.5cm}

\begin{rem*}
Our estimates of $\rho_{\alpha}(T)$ in Paragraph 2.2, of $\rho^{strong}(T)$
in Section 3, of $\rho_{\alpha}^{k}(T)$ in Paragraph 4.1 and of $\rho^{strong,\, k}(T)$
in Paragraph 4.2, hold for operators $T$ acting on a Banach space
$\left(E,\,\left|\cdot\right|\right)$ not necessarily of finite dimension
and not necessarily of Hilbert type, but with a finite spectrum $\sigma(T).$ \end{rem*}


\begin{thebibliography}{Z3}
{\normalsize \bibitem[BoSp]{key-2} N. Borovykh, M.N. Spijker, }\textit{\normalsize Resolvent
conditions and bounds on the powers of matrices, with relevance to
numerical stability of initial value problems,}{\normalsize {} Journ.
Comp. Appl. Math. 125 (2000) 41-56. }{\normalsize \par}

{\normalsize \bibitem[DS]{key-1-3} E. B. Davies and B. Simon, }\textit{\normalsize Eigenvalue
estimates for non-normal matrices and the zeros of random orthogonal
polynomials on the unit circle}{\normalsize , J. Approx. Theory 141-2,
(2006), 189\textendash{}213.}{\normalsize \par}

{\normalsize \bibitem[Dol]{key-1-2} E.P. Dolzhenko, }\textit{\normalsize Bounds
for derivatives of rational functions}{\normalsize , Izv. Akad. Nauk
SSSR Ser. Mat., 27 (1963), 9\textendash{}28, (Russian) }{\normalsize \par}

{\normalsize \bibitem[Dy]{key-3-1} K. M. Dyakonov, }\textit{\normalsize Smooth
functions in the range of a Hankel operator}{\normalsize , Indiana
Univ. Math. J. 43 (1994), 805-838. }{\normalsize \par}

{\normalsize \bibitem[GZ]{key-8} A. M. Gomilko, Y. Zemanek, }\textit{\normalsize On
the Uniform Kreiss Resolvent Condition}{\normalsize , Funkts. Anal.
Prilozh., 42:3 (2008), 81\textendash{}84 }{\normalsize \par}

{\normalsize \bibitem[Kr]{key-9} H. O. Kreiss, }\textit{\normalsize Über
die Stabilitätsdefinition für Differenzengleichungen die partielle
Differentialgleichungen approximieren}{\normalsize , BIT 2 (1962),
pp. 153-181 }{\normalsize \par}

{\normalsize \bibitem[LeTr]{key-1-1} R.J. Leveque, L.N Trefethen,
}\textit{\normalsize On the resolvent condition in the Kreiss matrix
theorem}{\normalsize , BIT 24 (1984), 584-591. }{\normalsize \par}

{\normalsize \bibitem[Nev]{key-5} O. Nevanlinna, }\textit{\normalsize On
the growth of the resolvent operators for power bounded operators,
in Linear Operators, Banach Center Publications,}{\normalsize {}
Volume 38, Inst. Math. Pol. Acad. Sciences (Warsaw) (1997), 247-264.}{\normalsize \par}

{\normalsize \bibitem[Nik]{key-1} N.Nikolski, }\textit{\normalsize Condition
Numbers of Large Matrices and Analytic Capacities,}{\normalsize {}
St. Petersburg Math. J., 17 (2006), 641-682.}{\normalsize \par}

{\normalsize \bibitem[Pek]{key-2-2} A. A. Pekarskii, }\textit{\normalsize Inequalities
of Bernstein type for derivatives of rational functions, and inverse
theorems of rational approximation,}{\normalsize {} Math. USSR-Sb.52
(1985), 557-574.}{\normalsize \par}

{\normalsize \bibitem[Sand]{key-3} J. Sand,}\textit{\normalsize {}
On some stability bounds subject to Hille-Yosida resolvent conditions}{\normalsize ,
BIT 36, 378-386 (1996).}{\normalsize \par}

{\normalsize \bibitem[Sp1]{key-4} M.N. Spijker, }\textit{\normalsize Numerical
stability, resolvent conditions and delay differential equations}{\normalsize ,
Appl. Numer. Math. 24 (1997), 233\textendash{}246. }{\normalsize \par}

{\normalsize \bibitem[Sp2]{key-2-1} M.N. Spijker, }\textit{\normalsize On
a conjecture by LeVeque and Trefethen related to the Kreiss matrix
theorem,}{\normalsize {} BIT 31 (1991), pp. 551\textendash{}555.}{\normalsize \par}

{\normalsize \bibitem[SpSt]{key-1} M. N . Spijker and F . A . J .
Straetemans, }\textit{\normalsize Stability estimates for families
of matrices of nonuniformly bounded order}{\normalsize , Linear Algebra
Appl., to appear .}{\normalsize \par}

{\normalsize \bibitem[Tad]{key-7} E. Tadmor, }\textit{\normalsize The
resolvent condition and uniform power boundedness}{\normalsize , Linear
Algebra Appl. 80 (1981), pp. 250\textendash{}252.}{\normalsize \par}

{\normalsize \bibitem[Z1]{key-1-5} R. Zarouf}\textit{\normalsize ,
Asymptotic sharpness of a Bernstein-type inequality for rational functions
in $H^{2},$}{\normalsize {} to appear in St. Petersburg. Math. Journal
(2009).}{\normalsize \par}

{\normalsize \bibitem[Z2]{key-1-6} R. Zarouf, }\textit{\normalsize A
resolvent estimate for operators with finite spectrum, }{\normalsize preprint.}{\normalsize \par}

{\normalsize \bibitem[Z3]{key-1-4} R. Zarouf,}\textit{\normalsize {}
Sharpening a result by E.B. Davies and B. Simon}{\normalsize , C.
R. Acad. Sci. Paris, Ser. I 347 (2009). }
\end{thebibliography}
\end{document}